\documentclass[a4paper]{article}
\usepackage{enumitem}
\usepackage{calc}

\usepackage{lmodern} %Type1-font for non-english texts and characters
\usepackage{graphicx} %%For loading graphic files

\usepackage{amsmath, accents}
\usepackage{amssymb,tikz}
\usepackage{pict2e}
\usepackage{amsthm}
\usepackage{amscd}
\usepackage{mathrsfs}
\usepackage{todonotes}
\usetikzlibrary{calc} % This is to add two points in tikzpictures

\usepackage{yfonts}

\usepackage{marvosym}
\usepackage[percent]{overpic}

\newtheorem{theorem}{Theorem} [section]

\newtheorem{lemma}[theorem]{Lemma}

\theoremstyle{definition}

\newcommand{\R}{\mathbb{R}}

\let\oldbibliography\thebibliography
\renewcommand{\thebibliography}[1]{\oldbibliography{#1}
\setlength{\itemsep}{-0.5pt}}

\def\XXint#1#2#3{{\setbox0=\hbox{$#1{#2#3}{\int}$}
\vcenter{\hbox{$#2#3$}}\kern-.5\wd0}}

\usepackage{tikz}
\usetikzlibrary{arrows}
\usetikzlibrary{decorations.pathmorphing}
\usetikzlibrary{decorations.markings}
\usetikzlibrary{patterns}
\usetikzlibrary{automata}
\usetikzlibrary{positioning}
\usepackage{tikz-cd}
\tikzset{->-/.style={decoration={
				markings,
				mark=at position #1 with {\arrow{latex}}},postaction={decorate}}}
	
	\tikzset{-<-/.style={decoration={
				markings,
				mark=at position #1 with {\arrowreversed{latex}}},postaction={decorate}}}

\usetikzlibrary{shapes.misc}\tikzset{cross/.style={cross out, draw, 
         minimum size=2*(#1-\pgflinewidth), 
         inner sep=0pt, outer sep=0pt}}
         
\tikzset{
	master/.style={
		execute at end picture={
			\coordinate (lower right) at (current bounding box.south east);
			\coordinate (upper left) at (current bounding box.north west);
		}
	},
	slave/.style={
		execute at end picture={
			\pgfresetboundingbox
			\path (upper left) rectangle (lower right);
		}
	}
}

\usepackage{pgfplots}
\allowdisplaybreaks	

\numberwithin{equation}{section}

\def\bigO{{\cal O}}

\makeatletter
\newcommand{\oset}[3][0ex]{%
  \mathrel{\mathop{#3}\limits^{
    \vbox to#1{\kern-2\ex@
    \hbox{$\scriptstyle#2$}\vss}}}}
\makeatother

\usepackage[colorlinks=true]{hyperref}
\hypersetup{urlcolor=blue, citecolor=red, linkcolor=blue}

\begin{document}
\title{A point process on the unit circle with antipodal interactions}
\author{Christophe Charlier}

\maketitle

\begin{abstract}
We introduce the point process 
\begin{align*}
\frac{1}{Z_{n}}\prod_{1 \leq j < k \leq n} |e^{i\theta_{j}}+e^{i\theta_{k}}|^{\beta}\prod_{j=1}^{n} d\theta_{j}, \qquad \theta_{1},\ldots,\theta_{n} \in (-\pi,\pi], \quad \beta > 0,
\end{align*}
where $Z_{n}$ is the normalization constant. This point process is \textit{attractive}: it involves $n$ dependent, uniformly distributed random variables on the unit circle that attract each other. (For comparison, the well-studied C$\beta$E involves $n$ uniformly distributed random variables on the unit circle that repel each other.)

We consider linear statistics of the form $\sum_{j=1}^{n}g(\theta_{j})$ as $n \to \infty$, where $g\in C^{1,q}$ and $2\pi$-periodic. We prove that the leading order fluctuations around the mean are of order $n$ and given by $\smash{\big(g(U)-\int_{-\pi}^{\pi}g(\theta) \frac{d\theta}{2\pi}}\big)n$, where $U \sim \mathrm{Uniform}(-\pi,\pi]$. We also prove that the subleading fluctuations around the mean are of order $\sqrt{n}$ and of the form $\mathcal{N}_{\mathbb{R}}(0,4g'(U)^{2}/\beta)\sqrt{n}$, i.e. that the subleading fluctuations are given by a Gaussian random variable that itself has a random variance.

%We also derive large $n$ asymptotics for $Z_{n}$ (and some generalizations), up to and including the term of order $1$.

Our proof uses techniques developed by McKay and Isaev \cite{McKay, IsaevMcKay} to obtain asymptotics of related $n$-fold integrals.
\end{abstract}


%\tableofcontents

\section{Introduction}

Gibbs measures are models for collections of locally
dependent random points (or particles) and are important in various problems in probability and statistical physics \cite{DVJ2008}. Interactions between particles can be either attractive or repulsive. A well-known repulsive Gibbs measure is the circular $\beta$-ensemble (C$\beta$E), given by
\begin{align}\label{CbetaE}
\frac{1}{\tilde{Z}_{n}}\prod_{1 \leq j < k \leq n} |e^{i\theta_{j}}-e^{i\theta_{k}}|^{\beta}\prod_{j=1}^{n} d\theta_{j}, \qquad \theta_{1},\ldots,\theta_{n} \in (-\pi,\pi],
\end{align}
where $\tilde{Z}_{n}$ is the normalization constant. Here the $n$ points are confined to the unit circle and repel each other according to the two-dimensional Coulomb law at inverse temperature $\beta>0$. In order to maximize the density of \eqref{CbetaE}, the $n$ points must be evenly spaced on the unit circle. This point process has been widely studied, see e.g. \cite[Chapter 2]{For}.

\medskip In comparison, little is known about attractive point processes on the unit circle, and the purpose of this paper is to initiate the study of such a process. More precisely, we are interested in the joint probability measure
\begin{align}\label{new point process intro}
\frac{1}{Z_{n}}\prod_{1 \leq j < k \leq n} |e^{i\theta_{j}}+e^{i\theta_{k}}|^{\beta}\prod_{j=1}^{n} d\theta_{j}, \qquad \theta_{1},\ldots,\theta_{n} \in (-\pi,\pi], 
\end{align}
where $Z_{n}$ is the normalization constant. This point process is indeed an attractive Gibbs measure, because the density of \eqref{new point process intro} is maximized for point configurations of the form $(e^{i\theta_{1}},\ldots,e^{i\theta_{n}})=(e^{i\theta}, \ldots, e^{i\theta})$ with $\theta \in (-\pi,\pi]$.

\medskip In view of \eqref{CbetaE} and \eqref{new point process intro}, it is also natural to consider
\begin{align}\label{Mirror-type}
\frac{1}{\widehat{Z}_{n}}\prod_{1 \leq j < k \leq n} |e^{i\theta_{j}}-e^{-i\theta_{k}}|^{\beta}\prod_{j=1}^{n} d\theta_{j}, \qquad \theta_{1},\ldots,\theta_{n} \in (-\pi,\pi], 
\end{align}
where $\widehat{Z}_{n}$ is the normalization constant. The point process \eqref{Mirror-type} is attractive, but also features a repulsion with the real line: for $n\geq 3$, only the point configurations $(e^{i\theta_{1}},\ldots,e^{i\theta_{n}})=(i,\ldots,i)$ and $(e^{i\theta_{1}},\ldots,e^{i\theta_{n}})=(-i,\ldots,-i)$ maximize the density of \eqref{Mirror-type}.

\medskip By rotational symmetry, the random variables $e^{i\theta_{1}},\ldots,e^{i\theta_{n}}$ of both \eqref{CbetaE} and \eqref{new point process intro} are uniformly distributed on the unit circle (but not independently distributed). In the case of \eqref{CbetaE}, these random variables repel each other, while in the case of \eqref{new point process intro} they attract each other. On the other hand, for the point process \eqref{Mirror-type}, the individual distributions of $e^{i\theta_{1}},\ldots,e^{i\theta_{n}}$ are not uniform if $n \geq 3$.

\medskip In addition to providing concrete examples of attractive point processes on the unit circle, the point processes \eqref{new point process intro} and \eqref{Mirror-type} are also valuable from a mathematical point of view, because they can be studied rigorously as $n\to \infty$ using results from \cite{McKay, IsaevMcKay} and \cite{McKayWormald}, respectively (we comment more on this below).

\medskip Both \eqref{new point process intro} and \eqref{Mirror-type} can also be seen as repulsive point processes of a new kind, where the points $e^{i\theta_{1}},\ldots,e^{i\theta_{n}}$ do not repel each other, but are repelled by some ``image points". Indeed, the points $e^{i\theta_{1}},\ldots,e^{i\theta_{n}}$ of \eqref{new point process intro} are repelled by the image points $-e^{i\theta_{1}},\ldots,-e^{i\theta_{n}}$ obtained by reflection across the origin, and the points $e^{i\theta_{1}},\ldots,e^{i\theta_{n}}$ of \eqref{Mirror-type} are repelled by the image points  $e^{-i\theta_{1}},\ldots,e^{-i\theta_{n}}$ obtained by reflection across the real line. For these reasons, we say that \eqref{new point process intro} is a point process ``with antipodal interactions", and that \eqref{Mirror-type} is a point process ``with mirror-type interactions" (where the real line plays the role of the mirror).

\medskip In this paper we focus on the point process \eqref{new point process intro} with antipodal interactions. The other point process \eqref{Mirror-type} is studied in the companion paper \cite{C ReflectionLine}. Further comparisons between \eqref{CbetaE}, \eqref{new point process intro} and \eqref{Mirror-type} are provided at the end of this section. 

\medskip Our first result shows that for large $n$, all points of \eqref{new point process intro} cluster together in an arc of length $\bigO(n^{-\frac{1}{2}+\epsilon})$ with overwhelming probability, see also Figure \ref{fig:antipodal}. More precisely, we have the following.
\begin{theorem}\label{thm:prob}
Fix $\beta > 0$. For any $\epsilon \in (0,\frac{1}{8})$, there exists $c>0$ such that, for all large enough $n$,
\begin{align*}
\mathbb{P}\bigg( |e^{i\theta_{j}}-e^{i\theta_{n}}|\leq n^{-\frac{1}{2}+\epsilon} \mbox{ for all } j  \in  \{1,\ldots,n-1\} \bigg)  \geq 1-e^{-cn^{2\epsilon}}.
\end{align*} 
\end{theorem}

\begin{figure}[h]
\begin{center}
\begin{tikzpicture}[master]
\node at (0,0) {\includegraphics[width=3.5cm]{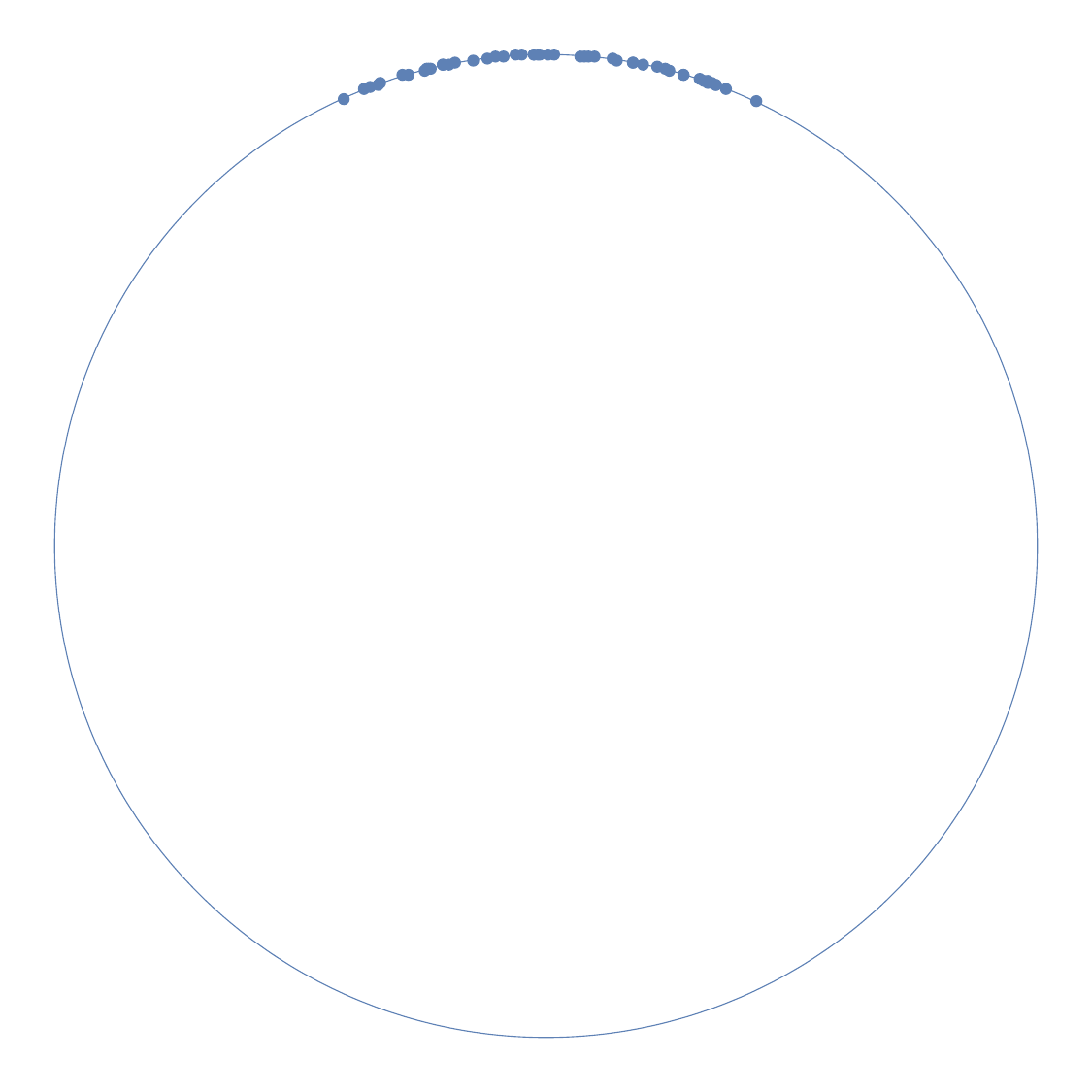}};
\draw[black,line width=0.15 mm,-<-=0.01,->-=1] ([shift=(60:1.35cm)]0,0) arc (60:120:1.35cm);
\node at (0,1) {$\bigO(n^{-\frac{1}{2}+\epsilon})$};
\end{tikzpicture} \hspace{0.5cm}
\begin{tikzpicture}[slave]
\node at (0,0) {\includegraphics[width=3.5cm]{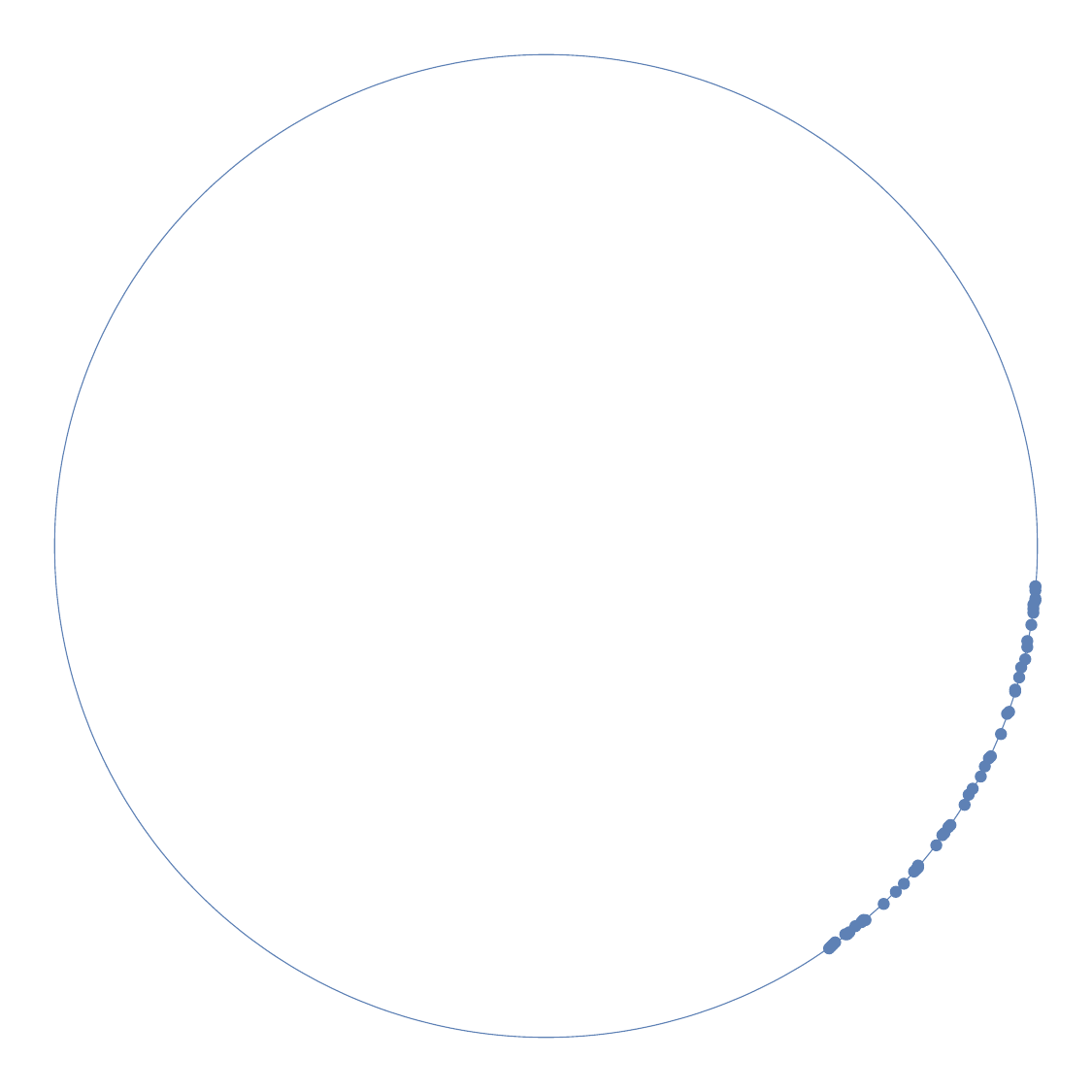}};
\draw[black,line width=0.15 mm,-<-=0.01,->-=1] ([shift=(-60:1.35cm)]0,0) arc (-60:-00:1.35cm);
\node at ($(-30:0.7)+(-0.1,0)$) {$\bigO(n^{-\frac{1}{2}+\epsilon})$};
\end{tikzpicture}
\caption{\label{fig:antipodal}Illustration of the point process \eqref{new point process intro} with $n=50$. With high probability all points are close to each other.}
\end{center}
\end{figure}

  In this paper, we study the fluctuations as $n \to \infty$ of linear statistics of the form $\sum_{j=1}^{n}g(\theta_{j})$, for fixed $\beta$ and where $g:\mathbb{R}\to \mathbb{R}$ is $2\pi$-periodic and sufficiently regular. More precisely, for Theorem \eqref{thm:almost sure}, $g$ is assumed to be continuous, but our other results (Theorems \ref{thm:main} and \ref{thm:moment generating function} below) require $g$ to be differentiable and such that $g'$ is H\"{o}lder continuous.

\medskip Let $\mu_{n} := \frac{1}{n}\sum_{j=1}^{n} \delta_{\theta_{j}}$ be the empirical measure of \eqref{new point process intro}. From the rotational symmetry of \eqref{new point process intro} together with Theorem \ref{thm:prob}, one expects the average empirical measure $\mathbb{E}[\mu_{n}]$ to converge as $n \to \infty$ to the uniform measure on $(-\pi,\pi]$ given by $\frac{d\theta}{2\pi}$. On the other hand, Theorem \ref{thm:prob} also implies that the support of $\mu_{n}$ will be contained inside an arc of length $\bigO(n^{-\frac{1}{2}+\epsilon})$ with overwhelming probability for large $n$. In other words, for large $n$ the measure $\mu_{n}$ ``deviates substantially" from $\mathbb{E}[\mu_{n}]$ with high probability, suggesting that $\mu_{n}$ has no deterministic limit as $n \to \infty$. The following result makes these ideas more precise.
\begin{theorem}\label{thm:almost sure}
Fix $\beta > 0$, and let $g:\mathbb{R}\to \mathbb{R}$ be $2\pi$-periodic and continuous. We have
\begin{align}\label{almost sure}
\int_{(-\pi,\pi]} g(x) d\mu_{n}(x) - g(\theta_{n}) \oset{\mathrm{a.s.}}{\underset{n\to \infty}{\xrightarrow{\hspace*{0.85cm}}}} 0.
\end{align}
Equivalently, $\frac{1}{n} \sum_{j=1}^{n} g(\theta_{j})- g(\theta_{n}) \oset{\mathrm{a.s.}}{\underset{n\to \infty}{\xrightarrow{\hspace*{0.85cm}}}} 0$. Since $\theta_{n} \sim \mathrm{Uniform}(-\pi,\pi]$, \eqref{almost sure} implies that
\begin{align}\label{Uniform one cut}
\int_{(-\pi,\pi]} g(x) d\mu_{n}(x) \oset{\mathrm{law}}{\underset{n\to \infty}{\xrightarrow{\hspace*{0.85cm}}}} \int_{(-\pi,\pi]} g(x) d\mu(x),
\end{align}
where $\mu := \delta_{U}$ and $U \sim \mathrm{Uniform}(-\pi,\pi]$. Equivalently, $\frac{1}{n} \sum_{j=1}^{n} g(\theta_{j})  \oset{\mathrm{law}}{\underset{n\to \infty}{\xrightarrow{\hspace*{0.85cm}}}} g(U)$.
\end{theorem}
\begin{proof}
The claim \eqref{almost sure} is a direct consequence of Theorem 1.1 and the Borel--Cantelli Lemma.
\end{proof}

\medskip Theorem \ref{thm:almost sure} deals only with the leading order fluctuations of $\sum_{j=1}^{n}g(\theta_{j})$. The subleading fluctuations are more intricate and will be given in Theorem \ref{thm:moment generating function} below. We will proceed by first establishing the large $n$ asymptotics of 
\begin{align}\label{main integral}
I(\tfrac{t}{\sqrt{n}}g) = \int_{-\pi}^{\pi} \dots \int_{-\pi}^{\pi}\prod_{1 \leq j < k \leq n} |e^{i\theta_{j}}+e^{i\theta_{k}}|^{\beta}\prod_{j=1}^{n} e^{\frac{t}{\sqrt{n}}g(\theta_{j})} d\theta_{j},
\end{align}
where $t$ lies in a compact subset of $\mathbb{R}$, see Theorem \ref{thm:main}. We will then derive the large $n$ asymptotics for the generating function of $\frac{1}{\sqrt{n}}\sum_{j=1}^{n}g(\theta_{j})$ simply from the formula 
\begin{align*}
\mathbb{E}\bigg[\exp\Big(\frac{t}{\sqrt{n}}\sum_{j=1}^{n} g(\theta_{j})\Big) \bigg] = \smash{\frac{I(\frac{t}{\sqrt{n}}g)}{I(0)}}.
\end{align*}
This paper is inspired by the works \cite{McKay,IsaevMcKay}. In the study of counting problems on graphs, McKay in \cite{McKay} introduced techniques to derive large $n$ asymptotics of several types of $n$-fold integrals, among which is
\begin{align}\label{nfold integral McKay and Wormald}
\frac{1}{(2\pi i)^{n}} \oint \dots \oint \frac{\prod_{1 \leq j <k \leq n}(z_{j}^{-1}z_{k}+z_{j}z_{k}^{-1})}{z_{1}z_{2}\dots z_{n}}dz_{1}\ldots dz_{n},
\end{align}
where each contour encloses the origin once anticlockwise. In recent years, a more systematic approach to such integrals was developed in \cite{IsaevMcKay}. The methods of \cite{McKay, IsaevMcKay} are remarkably robust and can be adjusted to analyze the integral \eqref{main integral} (despite the fact that the integrand in \eqref{main integral} is not analytic).

\medskip The following theorem gives a precise asymptotic formula, up to and including the constant term, for $I(\frac{t}{\sqrt{n}}g)$ as $n \to \infty$. %The special case $g=0$ of Theorem \ref{thm:main} establishes the asymptotics of $Z_{n}=I(0)$ as $n \to \infty$.
\begin{theorem}\label{thm:main}
Fix $\beta > 0$. Let $g:\mathbb{R}\to \mathbb{R}$ be $2\pi$-periodic and $C^{1,q}$, with $0<q\leq 1$, and let $t \in \mathbb{R}$. Then, for any fixed $0<\zeta< \frac{q}{2}$, as $n \to \infty$ we have
\begin{align}\label{asymptotic If}
I(\tfrac{t}{\sqrt{n}}g) = 2^{\beta \frac{n(n-1)}{2}} \bigg( \frac{8\pi}{\beta n} \bigg)^{\frac{n-1}{2}}\sqrt{n} \, e^{-\frac{1}{2\beta}} \big( 1+\bigO(n^{-\zeta}) \big) \int_{-\pi}^{\pi}  \exp \bigg( t\sqrt{n} g(\theta) + \frac{2 \, g'(\theta)^{2}}{\beta}t^{2} \bigg) d\theta,
\end{align}
uniformly for $t$ in compact subsets of $\mathbb{R}$. If $g\equiv 0$, then the error term can be improved: for any fixed $0<\zeta< 1$, we have
\begin{align}\label{asymptotic Zn}
Z_{n} = I(0) =  2^{\beta \frac{n(n-1)}{2}} \bigg( \frac{8\pi}{\beta n} \bigg)^{\frac{n-1}{2}}\sqrt{n} \, e^{-\frac{1}{2\beta}} 2\pi \big( 1+\bigO(n^{-\zeta}) \big), \qquad \mbox{as } n \to \infty.
\end{align}
\end{theorem}

Our next result on the generating function of $\frac{1}{\sqrt{n}}\sum_{j=1}^{n}g(\theta_{j})$ follows directly from Theorem \ref{thm:main}.
\begin{theorem}\label{thm:moment generating function}
Fix $\beta > 0$. Let $g:\mathbb{R}\to \mathbb{R}$ be $2\pi$-periodic and $C^{1,q}$, with $0<q\leq 1$, and let $t \in \mathbb{R}$. For any fixed $0<\zeta<\frac{q}{2}$, as $n \to \infty$ we have
\begin{align}
\mathbb{E}\bigg[e^{\frac{t}{\sqrt{n}} \sum_{j=1}^{n} g(\theta_{j})} \bigg] = \frac{I(\frac{t}{\sqrt{n}}g)}{I(0)} & = \frac{1+\bigO(n^{-\zeta})}{2\pi} \int_{-\pi}^{\pi}  \exp \bigg( t\sqrt{n} g(\theta) + \frac{2 \, g'(\theta)^{2}}{\beta}t^{2} \bigg) d\theta, \label{asymp for moment generating function}
\end{align}
uniformly for $t$ in compact subsets of $\mathbb{R}$.
\end{theorem}

The asymptotic formula \eqref{asymp for moment generating function} can be rewritten as
\begin{align}\label{lol14}
\mathbb{E}\bigg[e^{\frac{t}{\sqrt{n}}\sum_{j=1}^{n} g(\theta_{j})} \bigg] = \big( 1+o(1) \big)\mathbb{E}[e^{t [\sqrt{n} g(U) + \mathcal{N}_{\mathbb{R}}(0,\frac{4g'(U)^{2}}{\beta})]}],
\end{align}
where $U \sim \mathrm{Uniform}(-\pi,\pi]$, and where $\mathcal{N}_{\mathbb{R}}(0,\frac{4g'(u)^{2}}{\beta}):=0$ if $g'(u)=0$. Informally, one can interpret \eqref{lol14} as
\begin{align}\label{antipodal asymptotics lol}
\sum_{j=1}^{n} g(\theta_{j}) = n \, g(U) + \sqrt{n} \; N_{U} + o(\sqrt{n}), \qquad \mbox{where } N_{U}\sim \mathcal{N}_{\R}(0,\tfrac{4g'(U)^{2}}{\beta}).
\end{align}
Therefore, for a non-constant $g$, Theorem \ref{thm:moment generating function} means that the leading order fluctuations of $\sum_{j=1}^{n}g(\theta_{j})$ around the mean $n\int_{-\pi}^{\pi}g(\theta)\frac{d\theta}{2\pi}$ are of order $n$ and given by $n(g(U)-\int_{-\pi}^{\pi}g(\theta)\frac{d\theta}{2\pi})$. Moreover, the subleading fluctuations are of order $\sqrt{n}$ and given by $\smash{\mathcal{N}_{\mathbb{R}}(0,\frac{4g'(U)^{2}}{\beta})\sqrt{n}}$, i.e. by a Gaussian random variable whose variance is itself random. 

\paragraph{Other point processes on the unit circle.}
It is interesting to compare \eqref{new point process intro} with other point processes on the unit circle, such as
\begin{itemize}
\item[(a)] \vspace{-0.05cm} the C$\beta$E \eqref{CbetaE}, i.e. $\tilde{Z}_{n}^{-1}\prod_{j < k}|e^{i\theta_{j}}-e^{i\theta_{k}}|^{\beta}\prod_{j=1}^{n} d\theta_{j}$,
\item[(b)] \vspace{-0.15cm} the point process \eqref{Mirror-type}, i.e. $\widehat{Z}_{n}^{-1}\prod_{j < k} |e^{i\theta_{j}}-e^{-i\theta_{k}}|^{\beta}\prod_{j=1}^{n} d\theta_{j}$.
\end{itemize}
In sharp contrast with \eqref{new point process intro}, the empirical measure $\frac{1}{n}\sum_{j=1}^{n} \delta_{\theta_{j}}$ of the C$\beta$E converges almost surely to the uniform measure on $(-\pi,\pi]$, the associated smooth linear statistics have Gaussian fluctuations of order $1$, and the test function only affects the variance of this Gaussian. More informally, for the C$\beta$E we have
\begin{align}\label{lol16}
\sum_{j=1}^{n}g(\theta_{j}) = n \int_{-\pi}^{\pi} g(\theta) \frac{d\theta}{2\pi} + \mathcal{N}_{\mathbb{R}}(0,\sigma^{2}) + o(1), \qquad \mbox{as } n \to \infty
\end{align}
where $\sigma^{2} = \frac{4}{\beta} \sum_{k=1}^{\infty} k |g_{k}|^{2}$ and $g_{k} := \int_{-\pi}^{\pi}g(\theta)e^{-ik\theta}\frac{d\theta}{2\pi}$, see \cite{Jo1988}. Many other point processes with different types of repulsive interactions have been considered in the literature, see e.g. \cite{For, Boursier}. It is typically the case for such processes that (i) the associated empirical measures have deterministic limits, and (ii) the smooth statistics have Gaussian leading order fluctuations. The point process (b) listed above is very different: in fact, just like \eqref{new point process intro}, its empirical measure $\mu_{n}^{b}$ has no deterministic limit. Indeed, it is shown in \cite{C ReflectionLine} that if $g:\mathbb{R}\to \mathbb{R}$ is $2\pi$-periodic, bounded, and $C^{2,q}$ in neighborhoods of $\frac{\pi}{2}$ and $-\frac{\pi}{2}$ with $0<q\leq 1$, then
\begin{align*}%\label{Bernoulli one cut}
\int_{(-\pi,\pi]} g(x) d\mu_{n}^{b}(x) \oset{\mathrm{law}}{\underset{n\to \infty}{\xrightarrow{\hspace*{0.85cm}}}} \int_{(-\pi,\pi]} g(x) d\mu^{b}(x),
\end{align*}
where $\mu^{b} = B \delta_{\frac{\pi}{2}} + (1-B) \delta_{-\frac{\pi}{2}}$ and $B \sim \mathrm{Bernoulli} (\frac{1}{2})$ (i.e. $\mathbb{P}(B=1)=\mathbb{P}(B=0)=1/2$). In the generic case where $g(\frac{\pi}{2}) \neq g(-\frac{\pi}{2})$ and $g'(\frac{\pi}{2}) \neq g'(-\frac{\pi}{2})$, it is also proved in \cite{C ReflectionLine} that the leading order fluctuations of the smooth linear statistics of (b) are of order $n$ and purely Bernoulli, and that the subleading fluctuations are of order $1$ and of the form $BN_{1}+(1-B)N_{2}$, where $N_{1},N_{2}$ are two Gaussian random variables, and $N_{1},N_{2},B$ are independent from each other. Informally, the results from \cite{C ReflectionLine} can be rewritten as
\begin{align}\label{lol15}
\sum_{j=1}^{n} g(\theta_{j}) = n \Big(g(\tfrac{\pi}{2})B + g(-\tfrac{\pi}{2})(1-B)\Big) + BN_{1} + (1-B)N_{2} + o(1), \qquad \mbox{as } n \to \infty.
\end{align}
It is interesting to compare \eqref{antipodal asymptotics lol}, \eqref{lol16} and \eqref{lol15}. In particular, for the point processes (a) and (b), there are no fluctuations of order $\sqrt{n}$.

Another difference between (b) and \eqref{new point process intro} is the following: for (b), there are some non-generic test functions for which the leading order fluctuations around the mean are not of order $n$ (if $g(\frac{\pi}{2})=g(-\frac{\pi}{2})$), but of order $1$ or even of order $o(1)$. In comparison, for \eqref{new point process intro}, the only scenario where the leading order fluctuations are not of order $n$ corresponds to the trivial situation where $g$ is a constant, in which case there are no fluctuations at all.

\paragraph{Conclusion.} In this paper, we studied the smooth linear statistics of \eqref{new point process intro}. We proved formula \eqref{asymp for moment generating function} concerning the leading and subleading order fluctuations of $\sum_{j=1}^{n}g(\theta_{j})$. There are other problems concerning \eqref{new point process intro} that are also of interest for future research, for example:
\begin{itemize}
\item In this paper $\beta>0$ is fixed. The asymptotic formula \eqref{asymp for moment generating function} suggests that a critical transition occurs when $\beta \asymp n^{-\frac{1}{2}}$. It would be interesting to figure that out.
\end{itemize}

\section{Preliminary lemma}\label{section:prelim}
%In \cite{McKayWormald}, a method was developed to obtain large $n$ asymptotics of \eqref{nfold integral McKay and Wormald}. As mentioned, the method of \cite{McKayWormald} can actually be adapted to handle the analysis of \eqref{main integral}. 

Define
\begin{align*}
& U_{n}(t) = \{\boldsymbol{x} \in \mathbb{R}^{n}: |x_{i}| \leq t, i=1,\ldots,n\} \quad \mbox{for } t \geq 0.
\end{align*}
The following lemma is proved using techniques from \cite{McKay, IsaevMcKay} and will be used in Section \ref{section:proof} to obtain large $n$ asymptotics for $I(\frac{t}{\sqrt{n}}g)$.
\begin{lemma}\label{lemma:abcd}
Let $0<\epsilon<\frac{1}{8}$, $a>0$, $b \in \mathbb{R}$, $c \in \mathbb{R}$,  and $n \geq 2$ be an integer. Define
\begin{align*}
& J = \int_{U_{n-1}(n^{-\frac{1}{2}+\epsilon})} \exp \bigg( -a \sum_{1\leq j < k \leq n}(\theta_{j}-\theta_{k})^{2} + b \sum_{1\leq j < k \leq n}(\theta_{j}-\theta_{k})^{4} + \frac{c}{\sqrt{n}} \sum_{j=1}^{n-1}\theta_{j} \bigg) \prod_{j=1}^{n-1}d\theta_{j},
\end{align*}
where the integral is over $\boldsymbol{\theta}':=(\theta_{1},\ldots,\theta_{n-1}) \in U_{n-1}(n^{-\frac{1}{2}+\epsilon})$ with $\theta_{n} = 0$. Then, as $n \to \infty$,
\begin{align}\label{asymp of Jcal}
J = n^{1/2} \bigg( \frac{\pi}{a n} \bigg)^{\frac{n-1}{2}} \exp \bigg( \frac{c^{2}}{4a} + \frac{3b}{2a^{2}} + \bigO\big( n^{-1+8\epsilon} \big) \bigg).
\end{align}
\end{lemma}
\begin{proof}
If $c=0$, and if the error term $\bigO\big( n^{-1+8\epsilon} \big)$ in \eqref{asymp of Jcal} is replaced by the weaker estimate $\bigO\big( n^{-\frac{1}{2}+4\epsilon} \big)$, then the statement directly follows from \cite[Theorem 2.1]{McKay}. 

Changing variables $\theta_{j} = \alpha_{j}+\frac{c}{2a\sqrt{n}}$, $j=1,\ldots,n-1$, we get
\begin{align*}
J & = \exp \bigg( \frac{n-1}{n}\frac{c^{2}}{4a} + \frac{bc^{4}}{16a^{4}} \frac{n-1}{n^{2}} \bigg) \int_{-\frac{c}{2a\sqrt{n}}\mathbf{1}+U_{n-1}(n^{-\frac{1}{2}+\epsilon})} \exp \bigg( -a \sum_{1\leq j < k \leq n}(\alpha_{j}-\alpha_{k})^{2} \\
& + b \sum_{1\leq j < k \leq n}(\alpha_{j}-\alpha_{k})^{4} + \sum_{j=1}^{n-1}\bigg[ \frac{2bc}{a}\frac{\alpha_{j}^{3}}{\sqrt{n}} + \frac{3bc^{2}}{2a^{2}} \frac{\alpha_{j}^{2}}{n} + \frac{bc^{3}}{2a^{3}} \frac{\alpha_{j}}{n^{3/2}} \bigg] \bigg) \prod_{j=1}^{n-1}d\alpha_{j},
\end{align*}
where $\mathbf{1}=(1,\ldots,1)\in \mathbb{R}^{n-1}$ and $\alpha_{n}:=0$. For $(\alpha_{1},\ldots,\alpha_{n-1})\in -\frac{c}{2a\sqrt{n}}\mathbf{1}+U_{n-1}(n^{-\frac{1}{2}+\epsilon})$, we have
\begin{align*}
\sum_{j=1}^{n-1}\bigg[ \frac{2bc}{a}\frac{\alpha_{j}^{3}}{\sqrt{n}} + \frac{3bc^{2}}{2a^{2}} \frac{\alpha_{j}^{2}}{n} + \frac{bc^{3}}{2a^{3}} \frac{\alpha_{j}}{n^{3/2}} \bigg] = \bigO(n^{-1+3\epsilon}), \qquad \mbox{as } n \to \infty,
\end{align*}
and thus (resetting $\alpha_{j} = \theta_{j}$)
\begin{align}
J & = \exp \bigg( \frac{c^{2}}{4a} + \bigO(n^{-1+3\epsilon}) \bigg) \nonumber \\
& \times \int_{-\frac{c}{2a\sqrt{n}}\mathbf{1}+U_{n-1}(n^{-\frac{1}{2}+\epsilon})} \exp \bigg( -a \sum_{1\leq j < k \leq n}(\theta_{j}-\theta_{k})^{2}  + b \sum_{1\leq j < k \leq n}(\theta_{j}-\theta_{k})^{4} \bigg) \prod_{j=1}^{n-1}d\theta_{j}. \label{lol13}
\end{align}
Since the integrand in \eqref{lol13} is positive, and since
\begin{align*}
U_{n-1}(\tfrac{1}{2}n^{-\frac{1}{2}+\epsilon})\subset -\frac{c}{2a\sqrt{n}}\mathbf{1}+U_{n-1}(n^{-\frac{1}{2}+\epsilon}) \subset U_{n-1}(2n^{-\frac{1}{2}+\epsilon})
\end{align*}
holds for all sufficiently large $n$, to conclude the proof it remains to prove the claim for $c=0$.

Assume from now on that $c=0$. As mentioned, if $\bigO\big( n^{-1+8\epsilon} \big)$ in \eqref{asymp of Jcal} is replaced by $\bigO\big( n^{-\frac{1}{2}+4\epsilon} \big)$, then the statement follows from \cite[Theorem 2.1]{McKay}. The improved error term can be proved using the general method from \cite{IsaevMcKay}. Let $\mathbf{1}=(1,\ldots,1)^{T} \in \mathbb{R}^{n}$ and define
\begin{align*}
& \Omega = U_{n}(n^{-\frac{1}{2}+\epsilon}), \quad F(\boldsymbol{x}) = \exp \bigg( - a\sum_{1 \leq j<k\leq n}(x_{j}-x_{k})^{2} + b\sum_{1\leq j<k \leq n}(x_{j}-x_{k})^{4} \bigg), \\
& Q\boldsymbol{x} = \boldsymbol{x}-x_{n}\boldsymbol{1}, \quad W\boldsymbol{x} = \frac{\sqrt{a}}{\sqrt{n}}(x_{1}+\ldots+x_{n})\boldsymbol{1}, \quad P\boldsymbol{x}=\boldsymbol{x}-\frac{1}{n}(x_{1}+\ldots+x_{n})\boldsymbol{1}, \quad R\boldsymbol{x}=\frac{1}{\sqrt{a n}}\boldsymbol{x}.
\end{align*}
Clearly, $F(Q\boldsymbol{x})=F(\boldsymbol{x})$, $\dim \ker Q = 1$, $\dim \ker W=n-1$, $\ker Q \, \cap \, \ker W = \{\boldsymbol{0}\}$ and $\mathrm{span}(\ker Q, \ker W)=\mathbb{R}^{n}$. Applying \cite[Lemma 4.6]{IsaevMcKay} with $\rho = \sqrt{a}n^{\epsilon}$, $\rho_{1}=\rho_{2}=n^{-\frac{1}{2}+\epsilon}$, we obtain
\begin{align*}
J & = \int_{U_{n-1}(n^{-\frac{1}{2}+\epsilon})}F(\boldsymbol{\theta})d\boldsymbol{\theta}' = \int_{\Omega \cap Q(\mathbb{R}^{n})}F(\boldsymbol{y})d\boldsymbol{y} \\
& = (1-K)^{-1}\pi^{-\frac{1}{2}}\det(Q^{T}Q+W^{T}W)^{1/2} \int_{\Omega_{\rho}}F(\boldsymbol{x})e^{-\boldsymbol{x}^{T}W^{T}W\boldsymbol{x}}d\boldsymbol{x},
\end{align*}
where $\det(Q^{T}Q+W^{T}W)^{1/2}=\sqrt{a}n$,
\begin{align*}
& \Omega_{\rho} = \{\boldsymbol{x} \in \mathbb{R}^{n}:Q\boldsymbol{x} \in \Omega \mbox{ and } W\boldsymbol{x}\in U_{n}(\rho)\}, \\
& 0 \leq K < \min(1,ne^{-\frac{\rho^{2}}{\kappa^{2}}}) = ne^{-an^{1+2\epsilon}}, \qquad \kappa = \sup_{W \boldsymbol{x} \neq 0} \frac{\|W\boldsymbol{x}\|_{\infty}}{\|W\boldsymbol{x}\|_{2}} = \frac{1}{\sqrt{n}}.
\end{align*}
We thus have $J = (1+\bigO(e^{-c n^{1+2\epsilon}})) \frac{\sqrt{a}n}{\sqrt{\pi}} \int_{\Omega_{\rho}}F(\boldsymbol{x})e^{-a(x_{1}+\ldots+x_{n})^{2}}d\boldsymbol{x}$. Since
\begin{align*}
& \frac{\rho_{1}}{\|Q\|_{\infty}} = \frac{n^{-\frac{1}{2}+\epsilon}}{2}, \qquad \frac{\rho}{\|W\|_{\infty}} =  \frac{\sqrt{a}n^{\epsilon}}{\sqrt{a n}}= n^{-\frac{1}{2}+\epsilon}, \\
&  \|P\|_{\infty}\rho_{2} + \|R\|_{\infty} \rho = \frac{2(n-1)}{n} n^{-\frac{1}{2}+\epsilon} + \frac{1}{\sqrt{a n}}\sqrt{a}n^{\epsilon} \leq 3n^{-\frac{1}{2}+\epsilon},
\end{align*}
we also obtain from \cite[Lemma 4.6]{IsaevMcKay} that 
\begin{align*}
U_{n}(\tfrac{1}{2}n^{-\frac{1}{2}+\epsilon}) \subseteq \Omega_{\rho} \subseteq U_{n}(3n^{-\frac{1}{2}+\epsilon}).
\end{align*}
Furthermore, a direct computation gives
\begin{align*}
F(\boldsymbol{x})e^{-a(x_{1}+\ldots+x_{n})^{2}} = \exp \bigg( - an\sum_{j=1}^{n}x_{j}^{2} +b \sum_{1\leq j<k \leq n}(x_{j}-x_{k})^{4} \bigg).
\end{align*}
Let $\mathsf{f}(\boldsymbol{x})= b\sum_{j<k}(x_{j}-x_{k})^{4}$, and let $\boldsymbol{X}$ be a Gaussian random variable with density $(\frac{a n}{\pi})^{\frac{n}{2}}e^{-a n \boldsymbol{x}^{T}\boldsymbol{x}}$. Note that $|\partial \mathsf{f}/\partial x_{j}| = \bigO(n^{-\frac{1}{2}+3\epsilon})$ for $j=1,\ldots,n$ and $\boldsymbol{x} \in \Omega_{\rho}$, and that
\begin{align*}
\mathbb{E}[\mathsf{f}(\boldsymbol{X})] & = b \,  \mathbb{E}\bigg[ (n^{2}-4n)x_{1}^{4}+3 \bigg( \sum_{j=1}^{n}x_{j}^{2} \bigg)^{2} \bigg] = b \bigg( (n^{2}-n) \mathbb{E}\big[ x_{1}^{4} \big] + 3n(n-1) \mathbb{E}\big[x_{1}^{2}\big]^{2} \bigg) = \frac{3b(n-1)}{2 a^{2} n}.
\end{align*}
Applying \cite[Theorem 4.3]{IsaevMcKay} with $A=a nI$, $T=\frac{1}{\sqrt{a n}}I$, $\rho_{1}=\frac{\sqrt{a}}{2}n^{\epsilon},\rho_{2}=3\sqrt{a} \, n^{\epsilon}$, $\phi_{1},\phi_{2} \asymp n^{-\frac{1}{2}+4\epsilon}$ (with the functions $f,g$ and $h$ in \cite[Theorem 4.3]{IsaevMcKay} equal to $\mathsf{f}$, $\mathsf{f}$ and $0$ here, respectively, and with $\Omega$ in \cite[Theorem 4.3]{IsaevMcKay} equal to $\Omega_{\rho}$ here), we then find
\begin{align*}
J & = (1+\bigO(e^{-c n^{1+2\epsilon}})) \frac{\sqrt{a}n}{\sqrt{\pi}} \big( 1+\bigO(n^{-1+8\epsilon}) \big) \bigg( \frac{\pi}{a n} \bigg)^{\frac{n}{2}} e^{\frac{3b(n-1)}{2 a^{2} n}}, \qquad \mbox{as } n \to \infty,
\end{align*}
and \eqref{asymp of Jcal} follows.
\end{proof}

\section{Proof of Theorems \ref{thm:prob} and \ref{thm:main}}\label{section:proof}
We start with Theorem \ref{thm:main}. Our proof is inspired by \cite[Proof of Theorem 3.1]{McKay}. 

\noindent Using $|e^{i\theta_{j}}+e^{i\theta_{k}}|^{\beta}=2^{\beta}\big|\cos \tfrac{\theta_{j}-\theta_{k}}{2}\big|^{\beta}$, we first rewrite $I(\frac{t}{\sqrt{n}}g)$ as
\begin{align}\label{main integral proof}
I(\tfrac{t}{\sqrt{n}}g) = 2^{\beta \frac{n(n-1)}{2}}\int_{-\pi}^{\pi} \dots \int_{-\pi}^{\pi}\prod_{1 \leq j < k \leq n} |\cos \tfrac{\theta_{j}-\theta_{k}}{2}|^{\beta}\prod_{j=1}^{n} e^{\frac{t}{\sqrt{n}}g(\theta_{j})} d\theta_{j}.
\end{align}
%We will begin the evaluation of $I(\frac{t}{n}f)$ with the part of the integrand which will turn out to give the major contribution. 
Given $x \in \mathbb{R}$, let $x \,\mathrm{mod}\, 2\pi$ be such that $x \,\mathrm{mod}\, 2\pi = x + 2\pi k$, $k\in \mathbb{Z}$ and $x \,\mathrm{mod}\, 2\pi \in (-\pi,\pi]$. Fix $\theta_{n} \in (-\pi,\pi]$ and $\epsilon\in (0,\frac{1}{8})$, and let $I_{1}$ be the contribution to $I(\frac{t}{\sqrt{n}}g)$ of those $\boldsymbol{\theta}$ such that $|(\theta_{j}-\theta_{n}) \,\mathrm{mod}\, 2\pi|\leq n^{-\frac{1}{2}+\epsilon}$ for $1\leq j \leq n$, i.e.
\begin{align*}
& I_{1} := 2^{\beta \frac{n(n-1)}{2}}\int_{-\pi}^{\pi} e^{\frac{t}{\sqrt{n}}g(\theta_{n})} \bigg[ \int_{\theta_{n} + U_{n-1}(n^{-\frac{1}{2}+\epsilon})} \prod_{1 \leq j < k \leq n} |\cos \tfrac{\theta_{j}-\theta_{k}}{2}|^{\beta}\prod_{j=1}^{n-1} e^{\frac{t}{\sqrt{n}}g(\theta_{j})} d\theta_{j} \bigg]d\theta_{n},
\end{align*}
where $\theta_{n} + U_{n-1}(n^{-\frac{1}{2}+\epsilon})$ is equal to
\begin{align*}
\{\boldsymbol{\theta}'=(\theta_{1},\ldots,\theta_{n-1})\in (-\pi,\pi]^{n-1}: |(\theta_{j}-\theta_{n}) \,\mathrm{mod}\, 2\pi|\leq n^{-\frac{1}{2}+\epsilon} \mbox{ for } 1\leq j \leq n-1\}.
\end{align*}
Since $g \in C^{1,q}$, as $\theta \to \theta_{n}$ we have
\begin{align}
& \log\big[ 2^{\beta}\,|\hspace{-0.04cm}\cos \tfrac{\theta-\theta_{n}}{2}|^{\beta} \big] = \beta \log 2 - \frac{\beta}{8}(\theta-\theta_{n})^{2} - \frac{\beta}{192}(\theta-\theta_{n})^{4} + \bigO((\theta-\theta_{n})^{6}), \nonumber \\
& g(\theta) = g(\theta_{n}) + g'(\theta_{n})(\theta-\theta_{n}) + \bigO(|\theta-\theta_{n}|^{1+q}), \label{expansion of g}
\end{align}
where the implied constants are independent of $\theta_{n}\in (-\pi,\pi]$. For $\boldsymbol{\theta}'=(\theta_{1},\ldots,\theta_{n-1}) \in \theta_{n}+U_{n-1}(n^{-\frac{1}{2}+\epsilon})$, 
\begin{align}\label{some error in the proof}
& \bigO\bigg(\sum_{1\leq j<k \leq n}(\theta_{j}-\theta_{k})^{6}\bigg) = \bigO(n^{-1+6\epsilon}), \quad \bigO\bigg(\frac{1}{\sqrt{n}}\sum_{j=1}^{n-1} (\theta_{j}-\theta_{n})^{1+q}\bigg) = \bigO(n^{-\frac{q}{2}+(1+q) \epsilon}), \quad n \to \infty,
\end{align}
and thus 
\begin{multline*}
\prod_{1 \leq j < k \leq n} |\cos \tfrac{\theta_{j}-\theta_{k}}{2}|^{\beta}\prod_{j=1}^{n} e^{\frac{t}{\sqrt{n}}g(\theta_{j})} = e^{t \sqrt{n} \, g(\theta_{n})} \exp \bigg( - \frac{\beta}{8}\sum_{1 \leq j<k \leq n}(\theta_{j}-\theta_{k})^{2} \\
 - \frac{\beta}{192}\sum_{1 \leq j<k \leq n}(\theta_{j}-\theta_{k})^{4} + \frac{tg'(\theta_{n})}{\sqrt{n}}\sum_{j=1}^{n-1}(\theta_{j}-\theta_{n}) + \bigO(n^{-\frac{q}{2}+(1+q) \epsilon}+n^{-1+6\epsilon}) \bigg),
\end{multline*}
as $n \to \infty$, uniformly for $\theta_{n}\in (-\pi,\pi]$, for $\boldsymbol{\theta}' \in \theta_{n}+U_{n-1}(n^{-\frac{1}{2}+\epsilon})$, and for $t$ in compact subsets of $\mathbb{R}$. Because the integrand is positive, we then find
\begin{align}
I_{1} & = 2^{\beta \frac{n(n-1)}{2}}\exp \Big( \bigO(n^{-\frac{q}{2}+(1+q) \epsilon}+n^{-1+6\epsilon}) \Big) \int_{-\pi}^{\pi} e^{t \sqrt{n} \, g(\theta_{n})} I_{1}'(\theta_{n}) d\theta_{n}, \label{lol8} \\
I_{1}'(\theta_{n}) & := \int_{\theta_{n}+U_{n-1}(n^{-\frac{1}{2}+\epsilon})} \exp \bigg( - \frac{\beta}{8}\sum_{1 \leq j<k \leq n}(\theta_{j}-\theta_{k})^{2} - \frac{\beta}{192}\sum_{1 \leq j<k \leq n}(\theta_{j}-\theta_{k})^{4} \nonumber \\
& + \frac{tg'(\theta_{n})}{\sqrt{n}}\sum_{j=1}^{n-1}(\theta_{j}-\theta_{n}) \bigg) \prod_{j=1}^{n-1} d\theta_{j}, \nonumber
\end{align}
as $n \to \infty$ uniformly for $t$ in compact subsets of $\mathbb{R}$. %It is easy to see from the change of variables $\theta_{j} \to \theta_{j} + \theta_{n}$, $j=1,\ldots,n-1$ that $I_{1}'(\theta_{n})$ is independent of $\theta_{n}$. 
Using Lemma \ref{lemma:abcd} with $a=\frac{\beta}{8}$, $b=-\frac{\beta}{192}$ and $c=tg'(\theta_{n})$, we obtain
\begin{align}\label{I1prim asymp}
I_{1}'(\theta_{n}) = n^{\frac{1}{2}}\bigg( \frac{8\pi}{\beta n} \bigg)^{\frac{n-1}{2}} \exp \bigg( -\frac{1}{2\beta} + \frac{2t^{2}g'(\theta_{n})^{2}}{\beta} + \bigO(n^{-1+8\epsilon}) \bigg).
\end{align}
Substituting the above in \eqref{lol8} yields
\begin{align}
I_{1} & = 2^{\beta \frac{n(n-1)}{2}} \bigg( \frac{8\pi}{\beta n} \bigg)^{\frac{n-1}{2}} n^{\frac{1}{2}} e^{-\frac{1}{2\beta}} \exp \Big( \bigO(n^{-\frac{q}{2}+(1+q) \epsilon}+n^{-1+8\epsilon}) \Big) \nonumber \\
& \times \int_{-\pi}^{\pi} e^{t \sqrt{n} \, g(\theta_{n}) + \frac{2g'(\theta_{n})^{2}}{\beta}t^{2}} d\theta_{n}, \qquad \mbox{as } n \to \infty. \label{I1 asymp}
\end{align}
The rest of the proof consists in showing that $I(\frac{t}{\sqrt{n}}g)-I_{1}$ is negligible in comparison to $I_{1}$. Let $\tau = \pi/8$, and for $j \in \{-15,-14,\ldots,16\}$, define $A_{j} = ((j-1)\frac{\tau}{2},j\frac{\tau}{2}]$. For any $\boldsymbol{\theta} \in (-\pi,\pi]^{n}$, at least one of the $16$ intervals $A_{16}\cup A_{-15}, \; A_{-14}\cup A_{-13},  \ldots, A_{0}\cup A_{1},\ldots,A_{14}\cup A_{15}$ contains $\geq n/16$ of the $\theta_{j}$. Thus
\begin{align}\label{I tilde}
I(\tfrac{t}{\sqrt{n}}g)-I_{1} \leq 16 \widetilde{I}, \qquad \widetilde{I} := 2^{\beta \frac{n(n-1)}{2}} e^{ \sqrt{n} M(tg)}\int_{\mathcal{J}} \prod_{1 \leq j < k \leq n} |\cos \tfrac{\theta_{j}-\theta_{k}}{2}|^{\beta} \prod_{j=1}^{n} d\theta_{j},
\end{align}
where $M(tg) = \max_{\theta \in (-\pi,\pi]}tg(\theta)$ and 
\begin{align}\label{domain of I tilde}
\mathcal{J} = \{\boldsymbol{\theta}\in(-\pi,\pi]^{n}: \boldsymbol{\theta}' \notin \theta_{n}+U_{n-1}(n^{-\frac{1}{2}+\epsilon}) \mbox{ and } \# \{\theta_{j} \in A_{0}\cup A_{1}\} \geq \tfrac{n}{16} \}.
\end{align}
Define $S_{0}'=S_{0}'(\boldsymbol{\theta})$, $S_{1}'=S_{1}'(\boldsymbol{\theta})$ and $S_{2}'=S_{2}'(\boldsymbol{\theta})$ by
\begin{align*}
& S_{0}' = \{j : |\theta_{j}| \leq \tfrac{\tau}{2}\}, & & S_{1}' = \{j : \tfrac{\tau}{2} < |\theta_{j}| \leq \tau\}, & & S_{2}' = \{j: \tau<|\theta_{j}| \leq \pi\},
\end{align*}
and let $s_{2}' = \# S_{2}'$. If $j \in S_{2}'$ and $k \in S_{0}'$, then $|\cos \frac{\theta_{j}-\theta_{k}}{2}| \leq \cos(\tau/4)$. Thus the contribution to $\widetilde{I}$ of all the cases where $s_{2}' \geq n^{\epsilon}$ is at most 
\begin{align}\label{first bound}
2^{\beta \frac{n(n-1)}{2}}e^{ \sqrt{n} M(tg)} (\cos\tfrac{\tau}{4})^{\frac{\beta}{16}n^{1+\epsilon}}(2\pi)^{n} \leq \exp(-c_{1}n^{1+\epsilon})I_{1}
\end{align}
for some $c_{1}>0$ and for all sufficiently large $n$. Let us now define $S_{0}=S_{0}(\boldsymbol{\theta})$, $S_{1}=S_{1}(\boldsymbol{\theta})$ and $S_{2}=S_{2}(\boldsymbol{\theta})$ by
\begin{align*}
& S_{0} = \{j : |\theta_{j}| \leq \tau\}, & & S_{1} = \{j : \tau < |\theta_{j}| \leq 2\tau\}, & & S_{2} = \{j: 2\tau<|\theta_{j}| \leq \pi\},
\end{align*}
and let $s_{0} = \# S_{0}$, $s_{1} = \# S_{1}$ and $s_{2} = \# S_{2}$. The case $s_{2}'=s_{1}+s_{2} \geq n^{\epsilon}$ is already handled by \eqref{first bound}, and we now focus on the case $s_{2}'=s_{1}+s_{2}< n^{\epsilon}$. Let $I_{2}(m_{2})$ be the contribution to $\widetilde{I}$ of those $\boldsymbol{\theta}$ such that $s_{2}(\boldsymbol{\theta})=m_{2}$ and $s_{1}(\boldsymbol{\theta})+s_{2}(\boldsymbol{\theta})< n^{\epsilon}$. For $\boldsymbol{\theta} \in (-\pi,\pi]^{n}$, we have
\begin{align}\label{bounds}
\big|\cos \tfrac{\theta_{j}-\theta_{k}}{2}\big|^{\beta} \leq \begin{cases}
\exp(-\frac{\beta}{8}(\theta_{j}-\theta_{k})^{2}), & \mbox{if } j,k \in S_{0}\cup S_{1}, \\
(\cos \frac{\tau}{2})^{\beta}, & \mbox{if } j\in S_{0}, k\in S_{2}, \\
1, & %\mbox{otherwise},
\end{cases}
\end{align}
where we have used the fact that $|\cos \tfrac{x}{2}| \leq \exp (-\tfrac{x^{2}}{8})$ holds for all $|x| \leq \pi$. Thus
\begin{align*}
\prod_{1 \leq j < k \leq n} |\cos \tfrac{\theta_{j}-\theta_{k}}{2}|^{\beta} \leq \exp \bigg( -\frac{\beta}{8}\sum_{\substack{1 \leq j<k \leq n\\j,k \in S_{0}\cup S_{1}}}(\theta_{j}-\theta_{k})^{2} - \alpha s_{0} s_{2} \bigg)
\end{align*}
with $\alpha := -\beta \log \cos \frac{\tau}{2}$, and we find
\begin{align}
I_{2}(m_{2}) & \leq 2^{\beta \frac{n(n-1)}{2}} e^{ \sqrt{n} M(tg)} \binom{n}{m_{2}} \int_{|\theta_{1}|,\ldots,|\theta_{m_{2}}| \in (2\tau,\pi]}  \int_{\substack{|\theta_{m_{2}+1}|,\ldots,|\theta_{n}| \leq 2\tau \\ s_{0}(\boldsymbol{\theta}) \geq n-n^{\epsilon}}} \prod_{1 \leq j < k \leq n} |\cos \tfrac{\theta_{j}-\theta_{k}}{2}|^{\beta} \prod_{j=1}^{n} d\theta_{j} \nonumber \\
& \leq 2^{\beta \frac{n(n-1)}{2}} e^{\sqrt{n} \, M(tg)} e^{-\alpha m_{2}(n-n^{\epsilon})}(2\pi-4\tau)^{m_{2}} \binom{n}{m_{2}} I_{2}'(n-m_{2}), \label{lol11}
\end{align}
with
\begin{align}
& I_{2}'(m) = \int_{U_{m}(2\tau)} \exp \bigg( -\frac{\beta}{8}\sum_{1 \leq j<k \leq m}(\theta_{j}-\theta_{k})^{2} \bigg) \prod_{j=1}^{m}d\theta_{j} \leq 4\tau I_{2}''(m), \label{lol10} \\
& I_{2}''(m) = \int_{U_{m-1}(4\tau)} \exp \bigg( -\frac{\beta}{8}\sum_{1 \leq j<k \leq m}(\theta_{j}-\theta_{k})^{2} \bigg) \prod_{j=1}^{m-1}d\theta_{j}, \nonumber
\end{align}
and where in the definition of $I_{2}''(m)$ we set $\theta_{m}:=0$. We can estimate $I_{2}''(m)$ using the linear transformation $T$ of \cite[Proof of Theorem 2.1]{McKay}. This transformation $T:\mathbb{R}^{m-1}\to \mathbb{R}^{m-1}$ is defined by $T : (\theta_{1},\ldots,\theta_{m-1}) \mapsto \boldsymbol{y}=(y_{1},\ldots,y_{m-1})$,
where
\begin{align*}
y_{j} = \theta_{j} - \sum_{k=1}^{m-1}\frac{\theta_{k}}{m+m^{1/2}}, \qquad 1 \leq j \leq m-1.
\end{align*}
Let $\mu_{k} := \sum_{j=1}^{m-1}y_{j}^{k}$ for $k \geq 1$ and $V_{1} := T(U_{m-1}(4\tau))$. The following identities hold:
\begin{align}
& V_{1} = \{\boldsymbol{y}\in \mathbb{R}^{m-1} : |y_{j}+\mu_{1}/(m^{1/2}+1)| \leq 4\tau \mbox{ for } 1 \leq j \leq m-1 \}, \label{lol1 2} \\
& \mu_{1} = m^{-1/2} \sum_{j=1}^{m-1}\theta_{j}, \label{lol2 2} \\
& \sum_{1\leq j < k \leq m} (\theta_{j}-\theta_{k})^{2} = m \mu_{2}, \label{lol3 2} \\
& (\det T)^{-1} = m^{1/2}. \label{lol5 2}
\end{align}
%By \eqref{lol2 2}, $|\mu_{1}| \leq 4\tau \sqrt{m}$ holds for all $\boldsymbol{y} \in V_{1}$. Using also \eqref{lol1 2}, we obtain $V_{1} \subseteq U_{m-1}(8\tau)$, and then $\mu_{2} \leq (8\tau)^{2}m$. 
Hence, by \eqref{lol3 2} and \eqref{lol5 2}, and since the integrand of $I_{2}''(m)$ is positive,
\begin{align*}
I_{2}''(m) & = \sqrt{m}\int_{V_{1}} \exp \bigg( \hspace{-0.1cm} -\frac{\beta}{8} m\mu_{2} \bigg) \prod_{j=1}^{m-1}dy_{j} \leq \sqrt{m}\int_{\R^{m-1}} \hspace{-0.1cm} \exp \bigg( \hspace{-0.1cm}-\frac{\beta}{8}m\mu_{2} \bigg) \prod_{j=1}^{m-1}dy_{j} \leq \sqrt{m} \bigg( \frac{8\pi}{\beta m} \bigg)^{\frac{m-1}{2}}.
\end{align*}
Substituting the above in \eqref{lol10} and \eqref{lol11} yields $I_{2}'(m) \leq 4\tau \sqrt{m} \big( \frac{8\pi}{\beta m} \big)^{\frac{m-1}{2}}$ and
\begin{align*}
& I_{2}'(n-m_{2}) \leq e^{\bigO(m_{2}\log n)} \bigg( \frac{8\pi}{\beta n} \bigg)^{\frac{n-1}{2}}, \\
& I_{2}(m_{2}) \leq 2^{\beta \frac{n(n-1)}{2}} e^{ \sqrt{n} M(tg)} e^{-\alpha m_{2}(n-n^{\epsilon})}(2\pi-4\tau)^{m_{2}} n^{m_{2}}e^{\bigO(m_{2}\log n)} \bigg( \frac{8\pi}{\beta n} \bigg)^{\frac{n-1}{2}} \leq \exp\Big(-\frac{\alpha}{2} m_{2}n\Big)I_{1},
\end{align*}
for all sufficiently large $n$. Hence
\begin{align}\label{I2 bound}
\sum_{m_{2}=1}^{n^{\epsilon}}I_{2}(m_{2}) \leq \exp\Big(-c_{2}n\Big)I_{1}
\end{align}
for some $c_{2}>0$ and all large enough $n$. It only remains to analyze $I_{2}(0)$.  Let $I_{3}(\mathsf{h})$ be the contribution to $\widetilde{I}$ of those $\boldsymbol{\theta}$ such that 
\begin{itemize}
\item[(i)] $\#\{j:n^{-\frac{1}{2}+\epsilon} < |\theta_{j}-\theta_{n}| \leq 4\tau \} = \mathsf{h}$,
\item[(ii)] \vspace{-0.15cm} $\#\{j:|\theta_{j}-\theta_{n}| \leq n^{-\frac{1}{2}+\epsilon} \}=n-\mathsf{h}$, and
\item[(iii)] \vspace{-0.15cm} $|\theta_{n}| \leq 2\tau$.
\end{itemize}
By \eqref{I tilde} and \eqref{domain of I tilde}, and because the integrand is positive, we have $I_{2}(0) \leq \sum_{\mathsf{h}=1}^{n-1}I_{3}(\mathsf{h})$ and
\begin{align*}
I_{3}(\mathsf{h}) \leq 4\tau 2^{\beta \frac{n(n-1)}{2}}e^{ \sqrt{n} M(tg)} I_{3}'(\mathsf{h}), \quad \mbox{ with } \quad I_{3}'(\mathsf{h}) = \int_{\mathcal{J}_{\mathsf{h}}} \prod_{1 \leq j < k \leq n} |\cos \tfrac{\theta_{j}-\theta_{k}}{2}|^{\beta} \prod_{j=1}^{n-1}d\theta_{j},
\end{align*}
where in the definition of $I_{3}'(\mathsf{h})$ we set $\theta_{n}=0$ and
\begin{align*}
\mathcal{J}_{\mathsf{h}} = \{\boldsymbol{\theta}' \in U_{n-1}(4\tau): \#\{\theta_{j}:n^{-\frac{1}{2}+\epsilon} < |\theta_{j}| \leq 4\tau \} = \mathsf{h} \}.
\end{align*}
For $\boldsymbol{\theta}' \in U_{n-1}(4\tau)$, we have $\big|\cos \tfrac{\theta_{j}-\theta_{k}}{2}\big|^{\beta} \leq \exp(-\frac{\beta}{8}(\theta_{j}-\theta_{k})^{2})$ and thus 
\begin{align*}
I_{3}'(\mathsf{h}) \leq I_{3}''(\mathsf{h}), \quad \mbox{ where } \quad I_{3}''(\mathsf{h}) = \int_{\mathcal{J}_{\mathsf{h}}} \prod_{1 \leq j < k \leq n} \exp\bigg(-\frac{\beta}{8}(\theta_{j}-\theta_{k})^{2}\bigg) \prod_{j=1}^{n-1}d\theta_{j}.
\end{align*}
Applying the transformation $T$ defined above (but with $m$ replaced by $n$), we obtain
\begin{align*}
I_{3}''(\mathsf{h}) = \sqrt{n}\int_{T(\mathcal{J}_{\mathsf{h}})} \exp \bigg( \hspace{-0.1cm} -\frac{\beta}{8} n\mu_{2} \bigg) \prod_{j=1}^{n-1}dy_{j},
\end{align*}
where now $\mu_{k} := \sum_{j=1}^{n-1}y_{j}^{k}$, $k \geq 1$. The set $T(\mathcal{J}_{\mathsf{h}})$ consists of all $\boldsymbol{y} \in \mathbb{R}^{n-1}$ such that
\begin{itemize}
\item[(i)] $n^{-\frac{1}{2}+\epsilon}<|y_{j} + \frac{\mu_{1}}{\sqrt{n}+1}| \leq 4\tau$ for $\mathsf{h}$ values of $j$, and
\item[(ii)] \vspace{-0.15cm} $|y_{j} + \frac{\mu_{1}}{\sqrt{n}+1}| \leq n^{-\frac{1}{2}+\epsilon}$ for $n-1-\mathsf{h}$ values of $j$.
\end{itemize}
Since $\mathsf{h}\geq 1$, we easily conclude from (i) that any $\boldsymbol{y}\in T(\mathcal{J}_{\mathsf{h}})$ satisfies either $|\mu_{1}| > n^{\epsilon}/2$ or $|y_{j}| > n^{-\frac{1}{2}+\epsilon}/2$ for at least $\mathsf{h}$ values of $j$. Thus $\mu_{2}>\frac{1}{4}n^{-1+2\epsilon}$ holds for all $\boldsymbol{y}\in T(\mathcal{J}_{\mathsf{h}})$, which implies that
\begin{align*}
I_{3}''(\mathsf{h}) \leq \sqrt{n}\int_{\R^{n-1} \cap \{\mathbf{y}: \mu_{2}>\frac{1}{4}n^{-1+2\epsilon}\}} \exp \bigg( \hspace{-0.1cm} -\frac{\beta}{8} n\mu_{2} \bigg) \prod_{j=1}^{n-1}dy_{j} \leq \sqrt{n}\big( \frac{8\pi}{\beta n} \big)^{\frac{n-1}{2}} \exp(-c_{3}'n^{2\epsilon})
\end{align*}
for large $n$ and some $c_{3}'>0$, and we find
\begin{align}\label{I3 bound}
I_{2}(0) \leq \sum_{\mathsf{h}=1}^{n-1}I_{3}(\mathsf{h}) \leq 2^{\beta \frac{n(n-1)}{2}}e^{ \sqrt{n} M(tg)}\sqrt{n}\bigg( \frac{8\pi}{\beta n} \bigg)^{\frac{n-1}{2}} \exp(-2c_{3}n^{2\epsilon}) \leq \exp(-c_{3}n^{2\epsilon}) I_{1},
\end{align}
for some $c_{3}>0$ and all sufficiently large $n$. By \eqref{I tilde}, \eqref{first bound}, \eqref{I2 bound}, \eqref{I3 bound}, we have 
\begin{align}\label{lol1}
I(\tfrac{t}{\sqrt{n}}g)-I_{1} \leq \exp(-c_{4}n^{2\epsilon}) I_{1}
\end{align}
for some $c_{4}>0$ and all sufficiently large $n$. Theorem \ref{thm:main} for $g \not\equiv 0$ now follows from \eqref{I1 asymp}. From \eqref{expansion of g}, \eqref{some error in the proof} and \eqref{I1prim asymp}, we see that if $g \equiv 0$ then $\bigO(n^{-\frac{q}{2}+(1+q) \epsilon}+n^{-1+8\epsilon})$ in \eqref{I1 asymp} can be replaced by $\bigO(n^{-1+8\epsilon})$. This proves Theorem \ref{thm:main} for $g \equiv 0$.

\medskip Furthermore, for $g \equiv 0$, by definition of $I_{1}$ we have
\begin{align*}
\frac{I_{1}}{I(0)} & = \mathbb{P}\bigg( |\theta_{j}-\theta_{n}|\leq n^{-\frac{1}{2}+\epsilon} \mbox{ for all } j  \in  \{1,\ldots,n-1\} \bigg) \\
& \leq \mathbb{P}\bigg( |e^{i\theta_{j}}-e^{i\theta_{n}}|\leq n^{-\frac{1}{2}+\epsilon} \mbox{ for all } j  \in  \{1,\ldots,n-1\} \bigg).
\end{align*}
Hence Theorem \ref{thm:prob} directly follows from $I(0)-I_{1} \leq \exp(-c_{4}n^{2\epsilon}) I_{1}$ (which is \eqref{lol1} with $g\equiv 0$).

\medskip \noindent \textbf{Data availability statement.} There is no data associated to this work.

\medskip \noindent \textbf{Conflict of interest statement.} There is no conflict of interest.

\paragraph{Acknowledgements.} The author is grateful to Brendan McKay for useful remarks and for pointing out \cite{IsaevMcKay}. The author is also grateful to two anonymous referees for their careful reading and excellent remarks. Support is acknowledged from the Swedish Research Council, Grant No. 2021-04626.


\footnotesize
\begin{thebibliography}{99}
\bibitem{Bill} P. Billingsley,  Probability and measure. Anniversary edition. Wiley Series in Probability and Mathematical Statistics. John Wiley \& Sons, Inc., New York, 2012. 

\bibitem{Boursier} J. Boursier, Optimal local laws and CLT for 1D long-range Riesz gases, arXiv:2112.05881.

\bibitem{C ReflectionLine} C. Charlier, A point process on the unit circle with mirror-type interactions, arXiv:2212.06777.

\bibitem{DVJ2008} D.J. Daley and D. Vere-Jones, \textit{An introduction to the theory of point processes}, General theory and structure. Second edition Probab. Appl, Springer, New York, 2008.

\bibitem{For} P.J. Forrester, \textit{Log-gases and random matrices}, London Mathematical Society Monographs Series, 34. Princeton University Press, Princeton, NJ, 2010. 

\bibitem{IsaevMcKay} M. Isaev and B.D. McKay, Complex martingales and asymptotic enumeration, \textit{Random Structures Algorithms} \textbf{52} (2018), no. 4, 617--661. 

\bibitem{Jo1988} K. Johansson, On Szeg\H{o}'s asymptotic formula for Toeplitz determinants and generalizations, \textit{Bull. Sci. Math.} (2) \textbf{112} (1988), no. 3, 257--304.

\bibitem{McKay} B.D. McKay, The asymptotic numbers of regular tournaments, Eulerian digraphs and Eulerian oriented graphs,
\textit{Combinatorica} 10 (1990), no. 4, 367--377. 

\bibitem{McKayWormald} B.D. McKay and N.C. Wormald, Asymptotic enumeration by degree sequence of graphs of high degree, \textit{European J. Combin.} \textbf{11} (1990), no. 6, 565--580. 

\end{thebibliography}
\end{document}